%% file: ACC2018.tex
\documentclass[letterpaper, 10 pt, conference]{ieeeconf}  

\IEEEoverridecommandlockouts                              

\newtheorem{theorem}{Theorem}
\newtheorem{lemma}{Lemma}

\newtheorem{corollary}{Corollary}

\newtheorem{procedure}{Procedure}

\usepackage{comment}

\input Maths.tex

\usepackage[cmex10]{amsmath}
\usepackage{amsfonts}
\usepackage{amssymb}
\usepackage{mathtools}
\usepackage{nicefrac} 
\usepackage{mathrsfs}
\usepackage{cleveref,graphicx}
\usepackage[toc,acronym,nonumberlist]{glossaries}
\usepackage{enumerate,color}
\usepackage{overpic}

\usepackage{pgfplots} 
\usetikzlibrary{external}
\tikzsetexternalprefix{Figures/Tikz/}
\newcommand\includetikz[2][]{{{%
	\pgfkeys{/pgf/images/include external/.code={\includegraphics[#1]{##1}}}%
	\tikzsetnextfilename{#2}%
	\input{./Tikz/#2.tex}%
}}}
\renewcommand\includetikz[2][]{\includegraphics[#1]{Figures/Tikz/#2}}

\pgfplotsset{compat=newest} 

\usepackage{cleveref}
\usepackage{autonum}
\crefname{problem}{problem}{problems}
\crefname{proposition}{proposition}{propositions}
\crefname{procedure}{procedure}{procedures}
\crefname{assumption}{assumption}{assumptions}

\title{\LARGE \bf
Scale Free Bounds on the Amplification of Disturbances in Mass Chains
}

\author{Richard Pates and Kaoru Yamamoto
\thanks{The authors are members of the LCCC Linnaeus Center and the ELLIIT Excellence Center at Lund University.}
\thanks{This work was supported by the Swedish Research Council through the LCCC Linnaeus Center.}
\thanks{R. Pates and K. Yamamoto are  with the Department of Automatic Control, Lund University, Box 118, {SE-221 00 Lund}, Sweden.
        }%
}

\begin{document}

\maketitle
\thispagestyle{empty}
\pagestyle{empty}


\begin{abstract}
We give a method for designing a mechanical impedance to suppress the propagation of disturbances along a chain of masses. The key feature of our method is that it is \textit{scale free}. This means that it can be used to give a single, fixed, design, with provable performance guarantees in mass chains of any length. We illustrate the approach by designing a bidirectional control law in a vehicle platoon in a manner that is independent of the number of vehicles in the platoon.
\end{abstract}

\section{Introduction}



We study the propagation of disturbances in the mass chain in \Cref{fig:masschain}. More specifically we investigate how the transfer function from the position of the movable point to the first intermass displacement in a chain of $N$ masses
\begin{equation}\label{eq:Fn}
F_N\s\coloneqq{}\frac{x_0\s-x_1\s}{x_0\s}
\end{equation}
changes as $N$ changes. Our main contribution is to show that if
\begin{equation}\label{eq:h}
sZ\s{}m\eqqcolon{}h\s=\frac{a^2s^2}{\funof{1-a}s^2+2s+1},
\end{equation}
where $a>0$, then
\begin{equation}\label{eq:bound}
\sup_{N\in\mathbb{N}}\norm{F_N\s}_\infty\leq{}a.
\end{equation}
That is, when the models of the components in the mass chain take on the particular canonical form in \cref{eq:h}, the \Hfty{} norm of the the entire family of transfer functions $F_N\s$ can be bounded by $a$. We use this to develop a \textit{scale free} design method, that is a method that is independent of the number of masses in the chain, for designing the frequency responses of $F_N\s$.

The importance of a scale free design method for a mass chain stems from the prevalence of modelling network control problems with mechanical analogues. Examples include the platooning of vehicles \cite{LA66}, both frequency and voltage stability problems in electrical power systems \cite{MBB97,SDB16}, and flocking and consensus phenomena \cite{BJM12}. These applications typically involve very large numbers of subsystems, and the numbers of subsystems is often subject to change. The advantage of a scale free method is that it is easily applied independently of problem size, and any design remains valid even if the number of subsystems changes. 

The specific problem studied in this paper is too simplistic for many of these applications, since we only consider a single performance measure, and a string topology. Scale free performance criteria are however extremely rare. The overwhelming majority of existing scale free results, for example those based on passivity or dissipativity \cite{Wil72}, the multivariable Nyquist criterion \cite{LV06}, or IQC criteria \cite{PV17}, have focused on the question of \textit{robust stability}, rather than performance. This shortcoming stems at least in part from the fact that several key performance measures, in particular those relating to large scale network behaviours, simply do not scale. Notable examples include the string instability or network incoherence phenomena in \cite{SPH04,BJM12}. However it appears that average or local performance measures, for example those in \cite{Pat15,BJM12}, can be guaranteed in a scale free manner. The result we present here is another example. We therefore see the results in this paper as another step towards understanding the role of scale free design, in a setting that is relevant to a wide range of application areas.

The results in this paper build on the following recursive formula from \cite{YS16} for computing the transfer functions from the movable point to the first intermass displacement:
\begin{equation}\label{eq:it}
F_{N+1}\s=\frac{F_N\s+h\s}{F_N\s+h\s+1},\;F_0\s=0.
\end{equation}
The remarkable feature of this formula is that it allows the actual transfer function for $F_N\s$ to be computed in a simple manner, even when the impedance $Z\s$ has extremely general dynamics. This not only allows us to derive the norm based performance criterion in \cref{eq:bound}, given as \Cref{thm:main} in \Cref{sec:res}, but also to build up a picture of the entire frequency response of $F_N\s$. This, along with the conservatism of the method, is discussed in \Cref{sec:disc}.

\begin{figure}[!bt]
      \centering
      \includegraphics[trim=70mm 120mm 50mm 50mm, clip, scale=0.4]{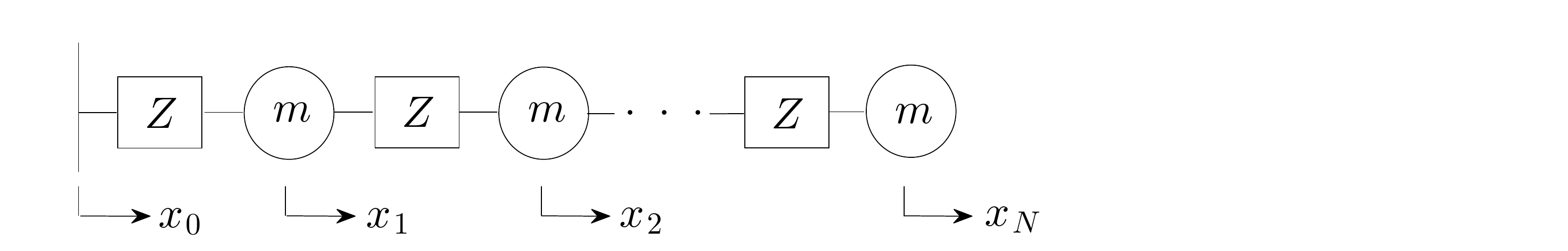}
      \caption{Chain of $N$ masses $m$ connected by a mechanical impedance $Z(s)$, and connected to a movable point $x_0$.}
      \label{fig:masschain}
\end{figure}

\section*{Notation}

$\Rat$ denotes the set of real rational not necessarily proper transfer functions, \Hfty{} denotes the Hardy space of transfer functions that are analytic on the open right half plane $\C_+$ with norm $\norm{G\s}_\infty\coloneqq{}\sup_{s\in\C_+}\abs{G\s}$, and $\RH\coloneqq{}\Rat\cap\Hfty$. 
$\Ce\coloneqq{}\C\cup\cfunof{\infty}$ denotes the extended complex plane. 
The principal value of the square root of $z \in \C$ is defined by $\sqrt{|z|}\exp\funof{j\angle z/2}$ where $-\pi<\angle z \leq \pi$.
The impedance of a linear time-invariant mechanical one-port network with force-velocity pair $(F, v)$ is defined by the ratio $v(s)/F(s)$.

\section{Results}\label{sec:res}

In this section we present the mathematical result that underpins the scale free design method discussed in \Cref{sec:disc}. The following theorem shows that when $h\s$ takes on the particular canonical form in \cref{eq:h}, the largest $\Hfty$ norm of the family of transfer functions generated by the complex iterative map in \cref{eq:it} can be bounded by $a$.

\begin{theorem}\label{thm:main}
Let $a>0$, and define the family of transfer functions
\[
F_{N+1}\s=\frac{F_N\s+h\s}{F_N\s+h\s+1},\;F_0\s=0.
\]
If
\[
h\s=\frac{a^2s^2}{\funof{1-a}s^2+2s+1},
\]
then
\[
\sup_{N\in\mathbb{N}}\norm{F_N\s}_\infty \leq a.
\]
\end{theorem}

\begin{proof}
The proof will be in two stages. We will first show that the following inequality holds for all $i\in\mathbb{N}$:
\begin{equation}\label{thm2eq:2}
\sup_{\omega\in\R}\abs{\frac{2}{a}F_i\jw-\sqrt{\frac{j\omega-1}{j\omega{}+1}}}\leq{}1.
\end{equation}
We will then show the above implies that $\norm{F_i\s}_\infty{}\leq{}a$.

We will prove that \cref{thm2eq:2} holds for all $i\in\mathbb{N}$ by induction. Define the all pass filter
\[
z\s\coloneqq{}\sqrt{\frac{s-1}{s+1}}.
\]
Therefore because $F_0\s=0$ and $\norm{z}_\infty=1$, \cref{thm2eq:2} holds for $i=0$. Now assume that \cref{thm2eq:2} holds for $i=N$. Define
\[
\begin{aligned}
n\s&=\funof{F_N\s+h\s},\\
m\s&=a\funof{F_N\s+h\s+1}.
\end{aligned}
\]
Hence $\frac{1}{a}F_{N+1}\s=\frac{n\s}{m\s}$. We may therefore rewrite \cref{thm2eq:2} for $i=N+1$ as
\begin{align}
\nonumber{}\funof{\frac{2n\jw}{m\jw}-z\jw}^*\!\!\funof{\frac{2n\jw}{m\jw}-z\jw}&\leq{}1\;\forall{\omega}\in\R.\label{thm2eq:4}
\end{align}
Multiplying through by $m\s^*m\s$ and using the fact that $z^*\jw{}z\jw=1$ shows that this is equivalent to
\[
\funof{2n^*n-z^*m^*n-zmn^*}\jw\leq{}0\;\forall{\omega}\in\R.
\]
Suppressing the dependence on $\omega$, it is quickly established that the above is equivalent to
\[
\begin{bmatrix}
F_{N}\\a
\end{bmatrix}^*
\funof{A+B+B^*}
\begin{bmatrix}
F_{N}\\a
\end{bmatrix}
\leq{}0,
\]
where
\[
A=\begin{bmatrix}
2&-z\\-z^*&0
\end{bmatrix}\;\text{and}\;B=\begin{bmatrix}
-az&\frac{2h}{a}-h\funof{z+z^*}\\0&\frac{\abs{h}^2}{a^2}-\frac{hz^*\funof{h^*+1}}{a}
\end{bmatrix}.
\]
Note that since \cref{thm2eq:2} holds for $i=N$,
\[
\begin{bmatrix}
F_{N}\\a
\end{bmatrix}^*
A
\begin{bmatrix}
F_{N}\\a
\end{bmatrix}
\leq{}0,
\]
so it is sufficient to show that $B+B^*\leq{}0$. Next observe that
\[
\begin{aligned}
z\jw&=\sqrt{\frac{1-\frac{1}{j\omega}}{1+\frac{1}{j\omega}}}
=\frac{1+\frac{j}{\omega}}{\sqrt{1+\frac{1}{\omega^2}}}\\
&=\exp\funof{j\arctan\funof{\frac{1}{\omega}}}.
\end{aligned}
\]
Defining $t\coloneqq{}\arctan\funof{\frac{1}{\omega}}$, and substituting in for $h$ and $z$, it can be verified algebraically that
\[
B+B^*=2a\cos{t}\begin{bmatrix}
-1&\frac{\omega^2\funof{a-\sec{t}}}{\omega^2\funof{a-1}+1+2j\omega{}}\\
\star{}&\frac{y}{\funof{\omega^2\funof{a-1}+1+2j\omega{}}^*\funof{\omega^2\funof{a-1}+1+2j\omega{}}}
\end{bmatrix}
\]
where the `$\star$' denotes the entry required to make the above Hermitian, and
\[
y=\omega^4\funof{-a^2+a-1}+a\omega^4\sec{t}-2\omega^3\tan{t}+\omega^2.
\]
Non-positivity of $B+B^*$ is equivalent to non-positivity of its Schur complement $\Delta$, which is given by:
\[
\begin{aligned}
\Delta={2a\cos{t}\funof{-1-{\omega^4}\funof{a-\sec{t}}^2/y}}.
\end{aligned}
\]
Substituting back in for $t$ shows that
\[
-y-\omega^4\funof{a-\sec{t}}^2=a\omega^4\funof{\sqrt{1/\omega^2+1}-1}\geq{}0.
\]
Observe that this also implies that $-y\geq{}\omega^4\funof{a-\sec{t}}^2$ which implies that $y\leq{}0$. Since $\cos{t}=1/\sqrt{1/\omega^2+1}$, it then follows that $\Delta\leq{}0$ as required. Therefore \cref{thm2eq:2} holds for $i=N+1$, and consequently for all $i\in\mathbb{N}$ by induction.

We will now show that this implies that $\norm{F_i\s}_\infty\leq{}a$. In words, \cref{thm2eq:2} states that for any ${\omega}\in\R$, the complex number $\frac{2}{a}F_{i}\jw$ lies within a circle centered on $z\jw$. Since $\abs{z\jw}=1$, it then follows that meeting \cref{thm2eq:2} implies that for all $i\in\mathbb{N}$
\begin{equation}\label{eq:111}
\sup_{\omega\in\R}\abs{\frac{2}{a}F_i\jw}\leq{}2,\;\Leftrightarrow{\;}\sup_{\omega\in\R}\abs{F_i\jw}\leq{}a.
\end{equation}
Finally, it is easily shown that $h\s\notin\funof{-4,0}$ for all $s$ in the closed right half plane. Therefore by \cite[Theorem 1]{YS16}, $F_i\in\RH$. Consequently $\norm{F_i\s}_\infty\leq{}a$ as required.
\end{proof}

\section{Discussion}\label{sec:disc}

\subsection{How can \Cref{thm:main} be used for scale free design?}

In this section we give a method for designing the mechanical impedance $Z\s$ to suppress disturbances from the movable point to the first intermass displacement in mass chains of any length. To do so, we pose the following weighted scale free \Hfty{} problem: Design $Z\s\in\Rat_Z$ such that
\begin{equation}\label{eq:desp}
\sup_{N\in\mathbb{N}}\norm{W\s{}F_N\s}_\infty\leq{}1.
\end{equation}
In the above $W\s\in\Rat$ is a weighting function, chosen to specify the requirements on the propagation of disturbances, and $\Rat_Z\subset\Rat$ the set of possible designs for the mechanical impedance $Z\s$. We will illustrate our method by designing a bidirectional controller for a vehicle platoon.

\begin{figure}
    \centering
   \includetikz[width=\columnwidth]{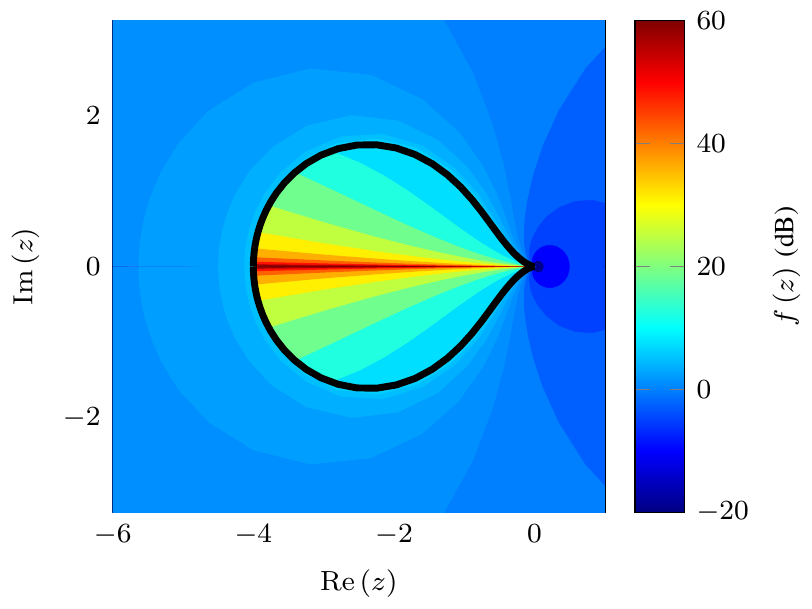}
    \caption{Plot of $f\funof{z}$. The black curve shows the contour $f\funof{z}=2$.}
    \label{fig:contour}
\end{figure}

\subsubsection{The scale free method} At first sight, \Cref{thm:main} looks far too restricted to solve the design problem in \cref{eq:desp}. After all, \Cref{thm:main} only applies when $h\s$ (and hence $Z\s$) takes on a particular canonical form, and doesn't involve any weighting functions. However as shown below, by exploiting the fact that an \Hfty{} norm bound guarantees a pointwise bound in $s$, \Cref{thm:main} can be used to bound $F_N\s$ even when general transfer functions $h\s$ are considered.

\begin{corollary}\label{cor:1}
Let $h\s\in\Rat$, and define
\begin{equation}\label{eq:fn}
\begin{aligned}
f\funof{z}\coloneqq{}&\min_{a>0,\omega\in\R\cup\cfunof{\infty}}a\\
&\text{s.t.}\;\;z=\frac{-a^2\omega^2}{-\funof{1-a}\omega^2+{2j}\omega+1}.
\end{aligned}
\end{equation}
For any $s\in \Ce$, if $h\s\notin\funof{-4,0}$, then
\[
\sup_{N\in\mathbb{N}}\abs{F_N\s}\leq{}f\funof{h\s}.
\]
\end{corollary}
\begin{proof}
If \cref{eq:fn} holds, then for some $s=j\omega$, a transfer function of the canonical form in \cref{eq:h} equals $z$. Hence for any $s$ such that $h\s=z$, $\abs{F_N\s}\leq{}a$ by \Cref{thm:main}. We perform the minimisation over $a$ because in general the pairs \funof{a,\omega} meeting \cref{eq:fn} are not unique. The result follows since it is easily established that the function is well defined for all $z\in\Ce\setminus\funof{-4,0}$.
\end{proof}

It follows from \Cref{cor:1} that we can tackle the problem in \cref{eq:desp} by designing an $h\s\in\Rat$ such that
\begin{equation}\label{eq:d2}
\max_{\omega\in\R\cup\cfunof{\infty}}\abs{W\jw}f\funof{h\jw}\leq{}1,
\end{equation}
and $h\s\notin\funof{-4,0},\forall{s}\in\C_+$. We can treat $f\funof{h\jw}$ just as we would a normal frequency response, and solve the problem by loopshaping. To help with this step, see \Cref{fig:contour}, which gives a contour plot of $f\funof{z}$. From this figure we see that $f\funof{z}$ gets larger and larger as we approach the `critical strip' \funof{-4,0}. Therefore in frequency ranges where a very small $f\funof{h\jw}$ is required, we only need to tune the parameters in $Z\jw$ to push $h\jw$ away from this region.

\begin{figure}
    \centering
   \includetikz[width=\columnwidth]{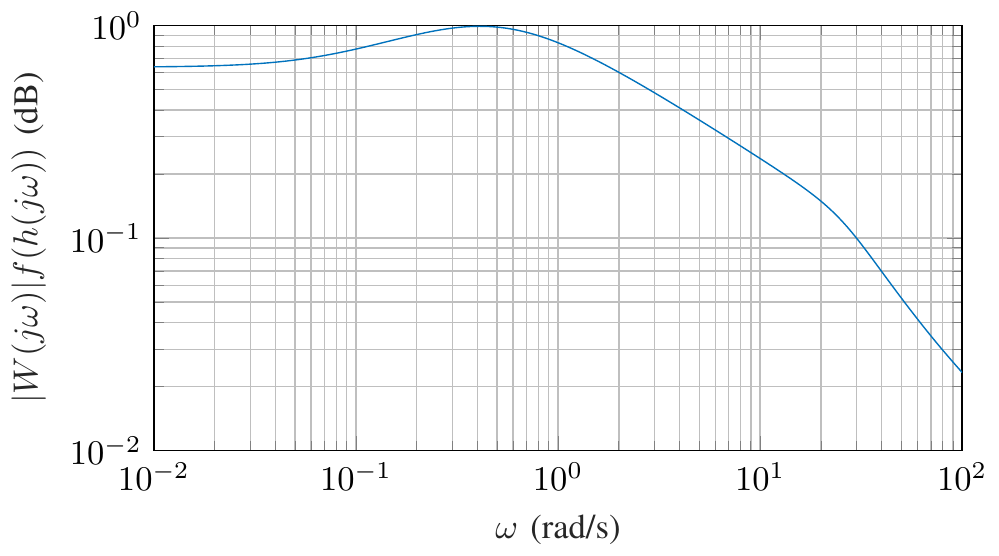}
    \caption{Plot of $\abs{W\jw}f\funof{h\jw}$ for the designed vehicle platoon controller.}
    \label{fig:bode}
\end{figure}

\subsubsection{A platooning example} When modelling a vehicle platoon as a mass chain, each mass is analogous to a car, and each impedance the dynamics of a symmetric bidirectional control scheme. The displacements $x_0$ and $x_i$ give the position of the lead vehicle and the position of the $i$th vehicle, respectively. We consider the problem of designing a bidirectional controller for a vehicle platoon, with the objective of suppressing the propagation of disturbances along the platoon as a result of the lead vehicle speeding up or slowing down. We assume that $m=1$ throughout, and that all variables are defined relative to a nominal constant velocity trajectory. 

We now put this problem into the form of \cref{eq:desp} by defining an appropriate weighting function and impedance parametrisation. To reflect the low frequency nature of the disturbance from the lead vehicle, we select
\[
W\s = \frac{9s}{4s+2}.
\]
The standard vehicle platoon model with symmetric bidirectional control (see e.g. \cite{SPH04}) can be compactly described by
\begin{equation}\label{eq:model}
\begin{aligned}
x_i\s&=\frac{1}{s^2\funof{0.1s+1}}u_i\s\\
u_i(s) &=\begin{cases}
K(s)((x_{i+1}(s)-x_i(s))-(x_i(s)-x_{i-1}(s)))\\
 \hfill\text{for } i=1,\dots, n-1, \\
K(s)(x_{i-1}(s) - x_i(s)) \hspace{2.85mm} \text{for } i=n.
\end{cases}
\end{aligned}
\end{equation}
In the above $u_i\s$ is the input to the \emph{i}th vehicle, and $K\s$ the controller to be designed. To bring the mass chain model in line with this, we define
\[
\Rat_Z=\cfunof{Z(s): Z(s)=\frac{s\funof{0.1s+1}}{K\s}, \, K\s\in\RH}.
\]
The objective is then to pick a $K\s\in\RH$ such that \cref{eq:d2} holds with $h\s=\frac{s^2\funof{0.1s+1}}{K\s}$. Designing $K\s$ by loopshaping resulted in
\[
K\s=\frac{\funof{s+0.05}\funof{s+10}}{\funof{0.01s+1}\funof{0.02s+1}}.
\]
The `frequency response' for this design is shown in \Cref{fig:bode}. Time domain simulations for this design in platoons of different lengths are shown in \Cref{fig:time}. This figure shows how the first inter-vehicle displacement changes in response to
\begin{equation}\label{eq:disturb}
\hat{x}_0\tm=\begin{cases}
\frac{1}{1+e^{-0.2t+8}}&\text{if $t\geq{}0$,}\\
0&\text{otherwise.}
\end{cases}
\end{equation}
Here $\hat{x}_0\tm$ gives the displacement of the fixed point in the time domain. Therefore \cref{eq:disturb} describes a scenario in which the lead vehicle undergoes a short period of acceleration before returning to its original velocity. The disturbance is well suppressed in all cases.

\begin{figure} 
    \centering 
 \includetikz[width=0.8\columnwidth]{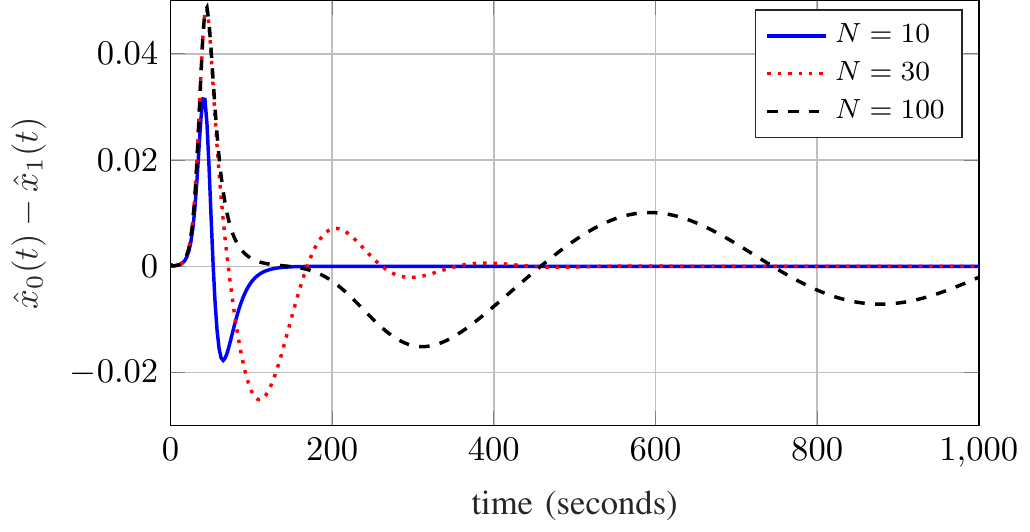}
    \caption{Change in first inter-vehicle displacement in response to \cref{eq:disturb}.}
    \label{fig:time}
\end{figure} 

\subsection{How conservative is \Cref{cor:1}?}

\begin{figure}
    \centering
   \includetikz{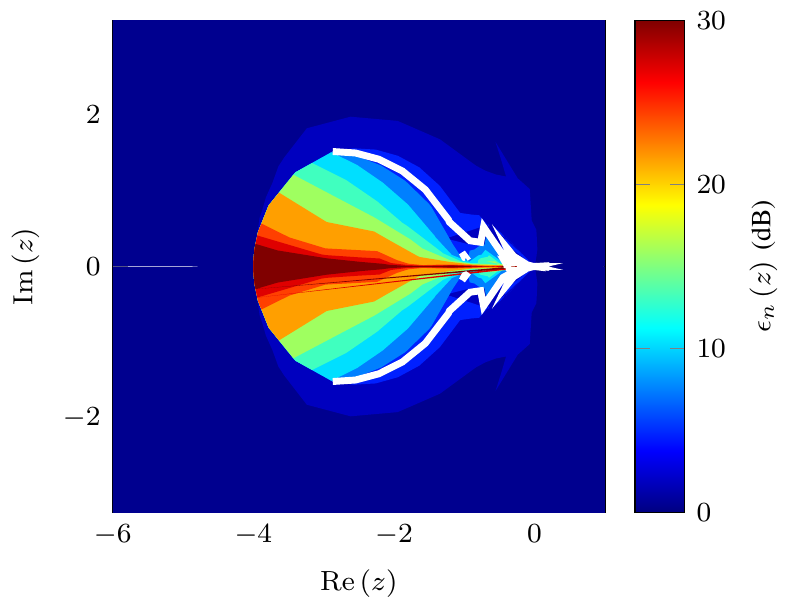}
    \caption{Plot of $\epsilon_n\funof{z}$, with $n=10^6$. The white curve shows the contour $\epsilon_n\funof{z}=2$.}
    \label{fig:err}
\end{figure}

In this section we will present numerical evidence to illustrate that provided $h\jw$  does not pass close to the critical strip \funof{-4,0}, the conservatism of our method is low. Consider the following lower bound.

\begin{lemma}
\label{lem:1}
Let $h\s\in\Rat$, and define 
\begin{equation}
    \begin{aligned}
    g_n\funof{z}\coloneqq{}&\max_{N\in\cfunof{1,\ldots{},n}}\abs{y_N}\\
    &\text{s.t.}\;\;y_{N+1}=\frac{y_N+z}{y_N+z+1},\, y_0=0.
    \end{aligned}
\end{equation}
For any $s\in\Ce$, if $h\s\notin\funof{-4,0}$, then
\[
\sup_{N\in\mathbb{N}}\abs{F_N\s}\geq{}g_n\funof{h\s}.
\]
\end{lemma}
\begin{proof}
By \cref{eq:it}, $F_N\s=y_N$. The result is now immediate, since $\sup_{N\in\mathbb{N}}\abs{y_N}\geq{}\max_{N\in\cfunof{1,\ldots{},n}}\abs{y_N}.$
\end{proof}

Combining \Cref{cor:1} and \Cref{lem:1} shows that
\[
g_n\funof{h\s}\leq{}\sup_{N\in\mathbb{N}}\abs{F_N\s}\leq{}f\funof{h\s}.
\]
Therefore if
\[
\epsilon_n\funof{h\s}\coloneqq{}\frac{f\funof{h\s}}{g_n\funof{h\s}}\approx{}1,
\]
then 
\begin{equation}\label{eq:approx}
f\funof{h\s}\approx{}\sup_{N\in\mathbb{N}}\abs{F_N\s}.
\end{equation}
\Cref{fig:err} plots $\epsilon_n\funof{z}$ for $n=10^6$. From this figure we see that for large regions of the complex plane, $\epsilon_n\funof{z}$ is approximately 1, and is rarely greater than 2. Furthermore, those regions where we have good agreement also correspond to the regions where $f\funof{z}$ is small, which is precisely what we are trying to achieve with our design. Therefore, as a design tool, the approximation proposed in \cref{eq:d2} seems a good one. Nevertheless, it does appear that this step does introduce conservatism, since based on \Cref{fig:err} it is doubtful that for all $z\in\Ce$, $\lim_{n\rightarrow{}\infty}g_n\funof{z}=f\funof{z}$. It is interesting to think how the arguments used in the proof of \Cref{thm:main} can be improved to reduce this conservatism.

\section{Conclusion}

A scale free design method for the suppression of disturbances in mass chains has been presented. The method allows for the design of a single mechanical impedance function that has provable performance guarantees in mass chains of any length. The approach was illustrated by designing a bidirectional controller for a vehicle platoon.

\bibliographystyle{IEEEtran}
\bibliography{references.bib}

\end{document}

%% file: Maths.tex
\newcommand{\abs}[1]{\ensuremath{\left\vert#1\right\vert}}

\newcommand{\cfunof}[1]{\ensuremath{\left\{#1\right\}}}

\newcommand{\C}{\ensuremath{\mathbb{C}}}

\newcommand{\Ce}{\ensuremath{\hat{\mathbb{C}}}}

\newcommand{\funof}[1]{\ensuremath{\left(#1\right)}}

\newcommand{\Hfty}{\ensuremath{\mathscr{H}_{\infty}}}

\newcommand{\jw}{\ensuremath{\left(j\omega\right)}}

\newcommand{\norm}[1]{\ensuremath{\left\Vert #1 \right\Vert}}

\newcommand{\Rat}{\ensuremath{\mathscr{R}}}

\newcommand{\RH}{\ensuremath{\mathscr{RH}_{\infty}}}

\newcommand{\R}{\ensuremath{\mathbb{R}}}

\newcommand{\s}{\ensuremath{\left(s\right)}}

\newcommand{\tm}{\ensuremath{\left(t\right)}}